\documentclass[a4paper, 10pt]{article}

\usepackage{a4wide}

\usepackage{amsmath, amssymb, graphicx}
\usepackage{hyperref}

\usepackage{amsthm}

\newtheorem{theorem}{Theorem}

\newtheorem*{corollary}{Corollary}

\theoremstyle{definition}

\usepackage{graphicx}
\usepackage[utf8]{inputenc}

\usepackage[T2A,T1]{fontenc}
\usepackage[russian,french,american]{babel}

\usepackage{dsfont}

\newcommand{\person}[1]{{#1}}

\newcommand{\N}{\mathds{N}}
\newcommand{\R}{\mathds{R}}

\newcommand{\defeq}{:=}

\newcommand{\ones}{\ensuremath{\mathbf{1}}}

\newcommand{\dout}{\ensuremath{d_\text{out}}}
\newcommand{\din}{\ensuremath{d_\text{in}}}

\DeclareMathOperator{\entrysum}{sum}
\newcommand{\colsum}{\ensuremath{c}}
\newcommand{\rowsum}{\ensuremath{r}}

\title{Chebyshev's Sum Inequality and the Zagreb Indices Inequality}
\author{Hanjo T{\"a}ubig%
  \thanks{%
  Computer Science Dept.,
  TU M\"unchen,
  D-85748 Garching,
  Germany,
  taeubig@in.tum.de}
}

\date{December 29, 2022}

\begin{document}

\maketitle

\begin{abstract}
% Avoid references to anything in the paper,
% since the abstract must be understandable for itself.
  In a recent article, Nadeem and Siddique used Chebyshev's sum
  inequality to establish the Zagreb indices inequality $M_1/n\le
  M_2/m$ for undirected graphs in the case where the degree sequence
  $(d_i)$ and the degree-sum sequence $(S_i)$ are similarly ordered.
  We show that this is actually not a completely new result and we
  discuss several related results that also cover similar inequalities
  for directed graphs, as well as sum-symmetric matrices and Eulerian
  directed graphs.
\end{abstract}

\section{Introduction}

\subsection{Notation}

We consider $n\times n$ matrices, denoted by~$A$, with
entries~$a_{ij}$.  In particular, we look at the total sum of entries
denoted by $\entrysum(A)$, as well as the row and column sums of~$A$,
which are denoted by $\rowsum_i(A)$ and $\colsum_j(A)$, respectively.
If $A$ is clear from the context, we abbreviate this by $\rowsum_i$
and~$\colsum_j$.
%
% We use $|c|$ for the absolute value (modulus) of $c\in\ComplNum$ and
% $|A|$ for the matrix where every single entry $a_{ij}$ of~$A$ is
% replaced by its modulus $|a_{ij}|$.
%
For the matrix power $A^p$, $p\in\N$, we define the following
abbreviations: $a_{ij}^{[p]}\defeq (A^p)_{ij}$,
$\rowsum_i^{[p]}\defeq\rowsum_i(A^p)$, and
$\colsum_j^{[p]}\defeq\colsum_j(A^p)$.
We assume that $A^0=I$ is the identity matrix.

As a special case, we consider adjacency matrices of directed and
undirected (multi-)graphs $G=(V,E)$ with $n\defeq |V|$ vertices and
$m\defeq |E|$ edges.
The in-degree and the out-degree of a vertex $v\in V$ are denoted
by~$\din(v)$ and $\dout(v)$, respectively.  In undirected graphs, the
degree of a vertex $v\in V$ is denoted by~$d(v)$.
A \emph{walk} in a multigraph~$G=(V,E)$ is an alternating sequence
$(v_0,e_1,v_1,\ldots,v_{k-1},e_k,v_k)$
of vertices $v_i\in V$ and edges $e_i\in E$
% \[
% (v_0,e_1,v_1,\ldots,v_{k-1},e_k,v_k)
% \qquad\text{ with }
% v_i\in V \text{ and }
% e_i\in E,
% \]
% (that starts and ends with a vertex)
% where each pair of consecutive vertices $(v_i,v_{i+1})$ must be
% connected by an edge $e_i=(v_i,v_{i+1})\in E$ of the graph.
where each edge~$e_i$ of the walk must connect vertex~$v_{i-1}$ to
vertex~$v_i$ in~$G$, that is, $e_i=(v_{i-1},v_i)$ for all
$i\in\{1,\ldots, k\}$.  Vertices and edges can be used repeatedly in
the same walk.  If the multigraph has no parallel edges, then the walks
could also be specified by the sequence of vertices
$(v_0,v_1,\ldots,v_{k-1},v_k)$ without the edges.  The \emph{length}
of a walk is the number of edge traversals.  That means, the walk
$(v_0,\ldots,v_k)$ consisting of $k+1$ vertices and $k$~edges is a
walk of length~$k$.  We call it a \emph{$k$-step walk}.
%
% Our main concern will be the investigation of the \emph{number of
%   walks} of a specified length.
%
Let $s_k(v)$ denote the number of $k$-step walks starting at vertex
$v\in V$ and let $e_k(v)$ denote the number of $k$-step walks ending
at~$v$.  If $G$ is undirected, then we have $w_k(v)\defeq s_k(v) =
e_k(v)$.  The total number of $k$-step walks is denoted by~$w_k$.
For walks of length~$0$, we have $s_0(v)=e_0(v)=1$ for each vertex~$v$
and $w_0=n$.
For walks of length~$1$, we have $s_1(v)=\dout(v)$ and
$e_1(v)=\din(v)$, i.e., $w_1(v)=d(v)$ for undirected graphs.  This
implies $w_1=\sum_{v\in V}\dout(v)=\sum_{v\in V}\din(v)=m$ for
directed graphs.  For undirected graphs, we have $w_1=\sum_{v\in
  V}d(v)=2m$ by the handshake lemma.

\subsection{\person{Chebyshev}'s Sum Inequality}
Two $n$-tuples $(a_1,\ldots,a_n)$ and $(b_1,\ldots,b_n)$ of real
numbers are called \emph{similarly ordered}%
\index{similarly ordered} if $(a_i-a_k)(b_i-b_k)\ge 0$ for all
$i,k\in[n]$.  They are called \emph{conversely ordered}
(also \emph{oppositely ordered}, see~\cite{hlp-i-59})
if $(a_i-a_k)(b_i-b_k)\le 0$ for all $i,k=1,\ldots,n$.
The term similarly ordered is equivalent to the requirement that there
exists a permutation that transforms both tuples into nonincreasing
sequences.  In the same line, two tuples are conversely ordered if and
only if there is a permutation that transforms one of the tuples into
a nonincreasing and the other tuple into a nondecreasing sequence.
Below, we will use the same notation for $n$-dimensional real vectors
$a,b\in\R^n$.

The following inequality was published by Chebyshev~\cite{c-sgubifi-1883,mv-hvgciqsp-74}.
\begin{theorem}[Chebyshev]%
  \label{thm:Chebyshev}
  Let $f,g: [a,b]\mapsto \R$ be integrable functions, both
  non-decreasing or both non-increasing. Furthermore, let
  $p:[a,b]\mapsto\R_{\ge 0}$ be an integrable nonnegative function.
  Then
\[
\int_a^b p(x) \,dx
\int_a^b p(x) f(x) g(x) \,dx
\ge
\int_a^b p(x) f(x) \,dx
\int_a^b p(x) g(x) \,dx\enspace.
\]
If one of the functions $f$ or $g$ is non-decreasing and the other
non-increasing, then the sign of inequality is reversed.
\end{theorem}

% see also \textcite[p.\,76]{s-csmc-04} % Steele
% for Chebyshev's inequality and connections to
% probability and statistics!

The discrete analog is the following statement.

% %
% % A generalization was considered in
% % \textcite[Lemma~3]{sce-ids-92}. % Székely, Clark, and Entringer
% %
% More general versions are discussed in
% % CAUTION! All sets are nonnegative in [HLP59]!
% % \cite[p.\,43]{hlp-i-59}, % Hardy, Littlewood, Pólya
% %
% the survey of Mitrinovi\'c and Vasi\'c~\cite{mv-hvgciqsp-74}, % Mitrinović and Vasić
% for instance the following variant that also appeared
% in a paper of Hayashi~\cite{h-osi-20}. % Hayashi
% % There is a mistake in this paper, see also
% % \textcite[p.\,14]{mv-hvgciqsp-74}, % Mitrinović and Vasić
% %
% % \textcite{h-osi-20} % Hayashi
% % also showed that it is sufficient that all $p_i$ have the same sign,
% % i.e., they might be either nonnegative or nonpositive.
% %
% % \textcite[p.\,18]{mv-hvgciqsp-74}, % Mitrinović and Vasić
% % cites \cite{hlp-i-59}, % Hardy, Littlewood, Pólya
% %
% % CAUTION!
% % From p.61 on, all $n$-tuples and sequences are POSITIVE
% % unless stated otherwise! (Convention 3)
% % and \cite[p.\,161]{b-hmti-03}. % Bullen
% %
% \begin{theorem}[\citeauthor{c-sgubifi-1883}]%
\begin{corollary}%
  % \label{thm:ChebyshevWeighted}
  For similarly ordered % $n$-tuples $(a_i)$ and $(b_i)$
  vectors $a\in\R^n$ and $b\in\R^n$
  and any
  nonnegative % $n$-tuple $(p_i)$
  vector $p\in\R_{\ge0}^n$, we have
  \[
  \left(\sum_{i=1}^n p_i a_i\right)
  \left(\sum_{i=1}^n p_i b_i\right)
  \le
  \left(\sum_{i=1}^n p_i\right)
  \left(\sum_{i=1}^n p_i a_i b_i\right)\enspace.
  \]
  % Equality holds if and only if all the $a_i$ or all the $b_i$ are
  % equal.\footnote{Note that the case of equality seems to be stated
  %   incorrectly as $(a_i)=(b_i)$ in \cite[p.\,2]{mv-hvgciqsp-74}.}
  % % 
  % \note{Is this really true?  Or does it hold only if all the $p_i$ are
  %   strictly positive?}
  % 
  The inequality is reversed if % $(a_i)$ and $(b_i)$
  $a$ and~$b$ are conversely ordered.
\end{corollary}
If $p\in\R_{\ge0}^n$ is nonzero, this corresponds to the following
weighted arithmetic means relation:
\[
\frac{\sum_{i=1}^n p_i a_i}{\sum_{i=1}^n p_i}\cdot\frac{\sum_{i=1}^n p_i b_i}{\sum_{i=1}^n p_i}
\le
\frac{\sum_{i=1}^n p_i a_i b_i}{\sum_{i=1}^n p_i}\enspace.
\]
A direct consequence is the following.
%
% (for a basic form where $r>0$ and $(a_i)$ and $(b_i)$ are \emph{nonnegative},
%  see \cite[p.\,43]{hlp-i-59}). % Hardy, Littlewood, and Pólya
% Let $\mathfrak{M}_r(a)=\left(\sum_{i=1}^n a_i^r\right)^{1/r}$.
% \[
%  \left(\frac{1}{n}\sum_{i=1}^n a_i^r\right)^{1/r}\cdot\left(\frac{1}{n}\sum_{i=1}^n b_i^r\right)^{1/r}
% \le
% \left(\frac{1}{n}\sum_{i=1}^n (a_i b_i)^r\right)^{1/r}\enspace.
% \]
%
Given $a,b\in\R^n$ and $r\in\R$, suppose that $a_i^r$ and $b_i^r$ are
defined within~$\R$ for all $i\in[n]$ and that the corresponding
tuples $(a_1^r,\ldots,a_n^r)$ and $(b_1^r,\ldots,b_n^r)$ are similarly
ordered.  Then we have
%
% If $(a_i^r)$ and $(b_i^r)$ are defined and similarly ordered (for
% $a_i,b_i,r\in\R$), then
% \[ \mathfrak{M}_r(a)\cdot\mathfrak{M}_r(b) < \mathfrak{M}_r(ab)\enspace, \]
\[
\frac{\sum_{i=1}^n p_i a_i^r}{\sum_{i=1}^n p_i}\cdot\frac{\sum_{i=1}^n p_i b_i^r}{\sum_{i=1}^n p_i}
\le
\frac{\sum_{i=1}^n p_i (a_i b_i)^r}{\sum_{i=1}^n p_i}\enspace.
\]
% Equality occurs if and only if all the $a_i^r$ or all the $b_i^r$ are
% equal.
% %
% \note{Is this really true?  Or does it hold only if all the $p_i$ are
%   strictly positive?  Check the case of arbitrary regular graphs, then
%   we have equality but not all eigenvalues are the same.  Does it mean
%   that the corresponding $p_i$'s are zero?}
%
% The inequality is reversed when $(a_i^r)$ and $(b_i^r)$ are
% conversely ordered.
%
One particular case where such inequalities can be obtained occurs for
\emph{arbitrary} exponents~$r$ and \emph{nonnegative} % $(a_i)$, $(b_i)$
vectors $a$ and~$b$ that are similarly or conversely ordered.  Another
special case is for \emph{odd integer} exponents~$r$ (or their reciprocals) and
\emph{arbitrary} real % tuples $(a_i)$ and~$(b_i)$.
vectors $a$ and~$b$.

\begin{corollary}%
  %\label{thm:ChebyshevSimple}
  If the vectors $a\in\R^n$ and $b\in\R^n$ are similarly ordered, then
  \[
  \left(\sum_{i=1}^n a_i\right) \left(\sum_{i=1}^n b_i\right)
  \le
  n \sum_{i=1}^n a_i b_i\enspace.
  \]
  The inequality is reversed if $a$ and~$b$ are conversely ordered.
  %
%   Equality holds if and only if all entries of~$a$ are equal or all
%   entries of~$b$ are equal.
\end{corollary}
For $n>0$, this is the same as the following relation between
arithmetic means:
\[
\frac{\sum_{i=1}^n a_i}{n}\cdot\frac{\sum_{i=1}^n b_i}{n}
\le
\frac{\sum_{i=1}^n a_i b_i}{n}\enspace.
\]
All those variants are called \person{Chebyshev}'s (sum) inequality.

\section{Zagreb Indices and Walks}

\subsection{The Zagreb Indices Inequality}
The first and the second \emph{Zagreb [group] index} for an undirected
graph $G=(V,E)$ are defined as%
\footnote{The first explicit definition of those indices appeared in
  the paper by Gutman et al.~\cite{grtw-gtmoap-75}.  Erroneously, it
  referred to the earlier article by Gutman and
  Trinajsti\'c~\cite{gt-gtmotpeeah-72} as the point where these
  measures where introduced.  Actually, this is not true.  This
  historical development was clarified recently by
  Gutman~\cite{g-otdbti-14}.}
\[
M_1=\sum_{v\in V}d_v^2\qquad\text{ and }\qquad
M_2=\sum_{\{x,y\}\in E}d_x d_y\enspace.
\]

Assume that $V=\{v_1,\ldots, v_n\}$ and that the vertex degrees are
abbreviated by $d_i=d(v_i)$.
Recently, an article was published by Nadeem and
Siddique~\cite{ns-mzii-22} that contains the following statement
concerning the degree-sums $S_i\defeq\sum_{v_j\in N(v_i)}d(v_j)$,
where $N(v_i)\defeq\{v_j\in V\mid\{v_i,v_j\}\in E\}$ is the set of
neighbors of~$v_i$.

\begin{theorem}%[\cite{ns-mzii-22}]
  \label{thm:NadeemSiddique}
  Let $G$ be a connected graph having degree sequence $(d_i)$,
  degree-sum sequence $(S_i)$, order~$n$ and size~$m$. If $(d_i)$ and
  $(S_i)$ are similarly ordered, then
  \[ \frac{M_1(G)}{n} \le \frac{M_2(G)}{m}\enspace. \]
  Equality is attained if and only if $G$ is a regular or a complete
  bipartite graph.
\end{theorem}
They also remark for the part with the sufficient condition, that the
Zagreb indices inequality holds for both, connected and non-connected
graphs.

That means, this result uses \person{Chebyshev}'s sum inequality to
establish the Zagreb indices inequality in the case where the
sequences $(d_i)$ and $(S_i)$ are similarly ordered.

\subsection{The Number of Walks Form}
For a long time during the research on topological indices in chemical
graph theory, it has been overlooked that two of the most popular
descriptors were in fact just special cases of measures defined by the
number of walks.
Only after decades, it was observed by Nikolić et
al.~\cite{nkmt-ziya-03} % Nikolic, Kovacevic, Milicevic, and Trinajstic
and Braun et
al.~\cite{bkmr-smdeziwc-05} % Braun, Kerber, Meringer, and Ruecker
that $M_1=w_2$ (which is also implicitly contained in the paper by
Gutman et al.~\cite{grr-wmg-01}, % Gutman, Rücker, and Rücker
but not explicitly stated there) and that $M_2=w_3/2$.

Together with $n=w_0$ and $m=w_1/2$, the Zagreb indices inequality can
be rephrased as
\[ w_2 / w_0 \le w_3 / w_1\enspace. \]

In the same line, we observe that the degree-sum $S_i$ equals the
number of $2$-step walks starting at~$v_i$, i.e., $S_i=w_2(v_i)$.  And
as already noted, we have $d_i=w_1(v_i)$.

In this respect, Theorem~\ref{thm:NadeemSiddique} can also be
expressed as a statement about walks:
\begin{theorem}
  Let $G$ be a graph having number of $1$-step walks sequence
  $(w_1(v_i))$ and number of $2$-step walks sequence $(w_2(v_i))$.
  % order~$n=w_0$ and size~$m=w_1/2$.
  %
  If $(w_1(v_i))$ and $(w_2(v_i))$ are similarly ordered, then
  \[ w_2 / w_0 \le w_3 / w_1\enspace. \]
\end{theorem}
Actually, this is not a new result.  It is a special case of a more
general theorem by Täubig \cite{t-impnwg-15,t-miis-17}, see the
corollary of Theorem~\ref{thm:Chebyshev_for_rowcolsums} in the next
section.
Note also that a related observation corresponding to the Zagreb
indices inequality has already been made by London~\cite{l-tiinsm-66}
in the more general case of entry sums of nonnegative symmetric
matrices.
% see also p.74 in the book

The Zagreb indices inequality has been shown to hold for several
special graph classes, such as
trees~\cite{vg-czifam-07,acs-nziitug-12}, chemical
graphs~\cite{hv-czi-07}, or subdivision
graphs~\cite{is-czi-09,t-inwsg-16}, while it does not hold for
connected graphs in general~\cite{lmsm-iwg-84,hv-czi-07} or for
bipartite graphs, not even for forests (see Chapter~5 of
\cite{t-impnwg-15} or~\cite{t-miis-17}).

%\section{\person{Chebyshev}'s Sum Inequality}%
\section{Applying Chebyshev's Sum Inequality to Directed Graphs}%
\label{sec:ChebyshevDirected}

In order to obtain inequalities for the number of walks in directed
graphs and for entry sums in nonsymmetric matrices, it is sometimes
possible to apply \person{Chebyshev}'s sum inequality (see
Theorem~\ref{thm:Chebyshev}).
In those cases we are able to obtain statements by elementary
% and more combinatorial
proofs without using any eigenvalues.
%
% \note{More generally: sum over walks.}
%
% If the walk numbers $e_k(x)$ and $s_k(x)$ are similarly ordered, then
% also the numbers $[e_k(x)-\frac{w_k}{n}]$ and $[s_k(x)-\frac{w_k}{n}]$
% are similarly ordered.
%
% It states that for two sequences $a_1\le\ldots\le a_n$ and
% $b_1\le\ldots\le b_n$ we have $\left(\sum_{k=1}^n a_k\right)
% \left(\sum_{k=1}^n b_k\right) \le n\sum_{k=1}^n a_kb_k$.
%
% \begin{owntheorem}%
\begin{theorem}%
  \label{thm:Chebyshev_for_rowcolsums}
  For any matrix~$A$ such that the column sums of~$A^k$ and the row
  sums of~$A^\ell$ (i.e., $\colsum^{[k]}$ and $\rowsum^{[\ell]}$)
  % $\colsum(A^k)$ and $\rowsum(A^\ell)$
  are similarly ordered, we have
  \[
  \entrysum\left(A^k\right)\cdot\entrysum\left(A^\ell\right)
  \le n\cdot\entrysum\left(A^{k+\ell}\right)\enspace.
  \]
  The inequality is reversed if $\colsum^{[k]}$ and $\rowsum^{[\ell]}$
  are conversely ordered.
% \end{owntheorem}
\end{theorem}
%
% A generalized weighted version does not yield what we expect.  We
% cannot use the simple Chebyshev inequality with a weight vector.  If
% we sum up the entries of $A^k\wv$ then this is different from
% $\entrysum_\wv(A^k)$ since we have the weighting only on one side(!)
% Hence, we need the weighted Chebyshev inequality.  But then, the
% terms (weighting) on the other side is not correct.
%
\begin{proof}
  For every $n\times n$ matrix~$A$, we have
  \[
  \entrysum\left(A^{k+\ell}\right)
  = \ones_n^T \left(A^k A^\ell\right) \ones_n
  = \left(\ones_n^T A^k\right) \left(A^\ell \ones_n\right)
  = \sum_{i\in [n]} \colsum^{[k]}_i\cdot\rowsum^{[\ell]}_i\enspace.
  \]
  The inequality is now a direct consequence of
  % Corollary~\ref{cor:Global2Decomp} and 
  \person{Chebyshev}'s inequality (see Theorem~\ref{thm:Chebyshev}):
  \begin{eqnarray*}
  \entrysum\left(A^k\right)\cdot\entrysum\left(A^\ell\right)
  & = &
  \left(\sum_{i=1}^n \colsum_i^{[k]}\right)
  \left(\sum_{i=1}^n \rowsum_i^{[\ell]}\right)\\
  & \le &
  n \sum_{i=1}^n \colsum_i^{[k]}\rowsum_i^{[\ell]}
  = n\cdot\entrysum\left(A^{k+\ell}\right)\enspace.
  \end{eqnarray*}
\end{proof}
%
% Since $w_{k+\ell}=\sum_{v\in V}e_k(v)\cdot s_\ell(v)$, we conclude the
% following.
%
Note that for all Hermitian matrices~$A$ and integers $k$, $\ell$
where $k+\ell$ is an even number,
Theorem~\ref{thm:Chebyshev_for_rowcolsums} holds in general without
the ordering assumption.  Those inequalities and related results for
real symmetric matrices and walks in undirected graphs were discussed
in \cite{twkhm-inwg-13} and \cite{tw-mpinwg-14}.

For the special case of adjacency matrices,
Theorem~\ref{thm:Chebyshev_for_rowcolsums} translates to the following
statement about the number of walks in digraphs.
% \begin{owncorollary}
\begin{corollary}
  \label{cor:Chebyshev_graph_walks}
%   If there is an ordering~$v_1,\ldots,v_n$ of the vertices such that
%   $e_k(v_i)$ and $s_\ell(v_i)$ are monotonically increasing (i.e.,
%   $e_k(v_i)\le e_k(v_{i+1})$ and $s_\ell(v_i)\le s_\ell(v_{i+1})$ for
%   all $i\in [1,n-1]$),
  For every directed graph $G=(V,E)$ where the vectors of walk numbers
  $e_k(v)$ and $s_\ell(v)$, $v\in V$, are similarly ordered, we have
  \[ w_k\cdot w_\ell\le n\cdot w_{k+\ell}\enspace. \]
% \end{owncorollary}
\end{corollary}
%
% This also means for undirected graphs that if some vertex ordering
% exists such that $w_k(v_i)$ and $w_\ell(v_i)$ are monotonically
% increasing sequences, then we have $nw_{k+\ell}\geq w_k w_\ell$.
Obviously, this inequality is applicable to undirected graphs if
$w_k(v_i)$ and $w_\ell(v_i)$, $i\in [n]$, are similarly ordered
sequences (here, we have $w_k(v_i)=s_k(v_i)=e_k(v_i)$ for all
$i,k\in\N$).  In particular, this is interesting if $k+\ell$ is an odd
number.

\paragraph{Inverted inequality:}
According to \person{Chebyshev}'s sum inequality (see
Theorem~\ref{thm:Chebyshev}), the inequality is inverted
% if there is a vertex ordering, such that one of the sequences
% $e_k(v_i)$ and $s_\ell(v_i)$ is monotonically increasing while the
% other one is monotonically decreasing.
if $e_k(v_i)$ and $s_\ell(v_i)$ are conversely ordered.
For instance, this would be applicable for $k=\ell=1$ if for each
vertex either the in-degree%
\index{in-degree}\index{degree!in-degree} or the out-degree%
\index{out-degree}\index{degree!out-degree} is equal to~$1$ and the
other one is greater or equal to~$1$.  Another example would be the
class of graphs where all vertices have the same sum of the in-degree
and the out-degree (that is, the same total degree).

\paragraph{Sum-symmetric matrices:}
From Theorem~\ref{thm:Chebyshev_for_rowcolsums}, we obtain a special
case if the row sums and the column sums of a matrix are similarly
ordered.  This happens, for example, in the case of sum-symmetric
matrices, i.e., if $\rowsum_i(A)=\colsum_i(A)$ for all~$i\in [n]$.
% see \textcite[p.\,184]{br-nma-97} % Bapat and Raghavan
% \begin{owncorollary}
\begin{corollary}
  For any sum-symmetric matrix~$A$%
  \index{sum-symmetric matrix}\index{matrix!sum-symmetric},
  we have
  \[ \entrysum(A)^2 \le n\cdot\entrysum(A^2)\enspace. \]
% \end{owncorollary}
\end{corollary}
Note that this corollary also follows from
% an inequality similar to Lemma~\ref{lem:powerJensen}.  Since
% $f(x)=x^2$ is convex in~$\R$, the assumption $a_i\ge 0$ in
% Lemma~\ref{lem:powerJensen} is not necessary for $p=2$.
Cauchy's inequality:
\[
\entrysum(A)^2
\ =\ \left(\sum_{i=1}^n\rowsum_i\right)^2
\ \le\ n\sum_{i=1}^n\rowsum_i^2
\ =\ n\sum_{i=1}^n\rowsum_i \colsum_i
\ =\ n\cdot\entrysum(A^2)\enspace.
\]

\paragraph{Eulerian directed graphs:}
We can apply this result to directed graphs as follows.
If there is a vertex ordering which is monotonically increasing with
respect to the in- and out-degrees, then the graph obeys the
inequality $nw_2\geq w_1^2$.  For instance, this is true if the
in-degree of each vertex equals its out-degree.
% , i.e., $\din(v)=\dout(v)$ for all $v\in V$.
%
% Here, in contrast to undirected graphs and their bidirected
% counterparts, incoming and outgoing edges do not need to end at the
% same vertices.
%
% \begin{owncorollary}
\begin{corollary}
  For every Eulerian directed graph%
  \index{Eulerian graph}\index{graph!Eulerian}
  ($\forall v\in V: \din(v)=\dout(v)$),
  % for all $v\in V$
  we have
  \[ w_1^2 \le n\cdot w_2\qquad\text{or}\qquad w_1/w_0 \le w_2/w_1\enspace. \]
% \end{owncorollary}
\end{corollary}
%
% Note that this inequality also follows from
% Lemma~\ref{lem:powerJensen} since
% $w_1^2 = (\sum_{v\in V}\din(v))^2 \le n\sum_{v\in V}\din(v)^2 = n\sum_{v\in V}\din(v)\dout(v)=nw_2$.

\end{document}